\documentclass[11pt]{amsart}

\usepackage{amssymb,amsthm,amsmath}
\usepackage{enumerate}

\usepackage{tikz}
\usepackage{tikz-network}

\usetikzlibrary{shapes.geometric}
\usetikzlibrary{calc}

\usepackage[latin2,utf8]{inputenc}
\usepackage{xcolor}
\usepackage[a4paper,margin=1in,footskip=0.25in]{geometry}

\newcommand{\R}{\mathbb{R}}
\renewcommand{\div}{{\rm div}\,}

\newcommand{\vx}{\mathbf{x}}
\newcommand{\vv}{\mathbf{v}}
\newcommand{\dd}{{\rm d}}
\newcommand{\intr}{\int_{\R^{2d}}}

\newcommand{\N}{{\mathcal N}}

\newtheorem{prop}{Proposition}
\newtheorem{lem}{Lemma}
\newtheorem{rem}{Remark}
\newtheorem{defi}{Definition}
\newtheorem{theo}{Theorem}
\newtheorem{example}{Example}

\title[]{A fuzzy $q$-closest alignment model}

\author{Piotr B. Mucha \& Jan Peszek}

\address{Institute of Applied Mathematics and Mechanics, University of Warsaw, ul. Banacha 2, 02-097 Warszawa, Poland}
\email{p.mucha@mimuw.edu.pl, j.peszek@mimuw.edu.pl}

\begin{document}

\date{\today}

\keywords{Cucker-Smale model, Cucker-Dong model, Alignment models, $q$-closest neighbors, Mean-field limit, Measure-valued solutions}
\subjclass[2020]{35B40, 35D30, 35L81, 35Q70, 35Q83} 

\thanks{\textbf{Acknowledgement:} The paper has been partly supported by the Polish National Science Centre’s Grant No. 2018/30/M/ST1/00340 (HARMONIA). JP's work has been additionally supported by the Polish National Science Centre's Grant No. 2018/31/D/ST1/02313 (SONATA)}

\begin{abstract}
    The paper examines the problems related to the well-posedness of the Cucker-Smale model with communication restricted to the $q$-closest neighbors, known also as the Cucker-Dong model. With agents oscillating on the boundary of different clusters, the system becomes difficult to precisely define, which leads to further problems with kinetic limits as the number of agents tends to infinity. We introduce the fuzzy $q$-closest system, which circumvents the issues with well-posedness. For such a system we prove a stability estimate for measure-valued solutions and perform the kinetic mean-field limit.
\end{abstract}
\maketitle

\section{Introduction}

The study of interacting particle systems is one of the prolific subjects of 
mathematical modelling, with a wide spectrum of applications ranging from aggregation of bacteria and flocking of birds to opinion dynamics, distribution of goods and traffic flows. Our main interest is the Cucker-Smale (CS) flocking model \cite{CS-07}, which describes the dynamics of an ensemble of autonomous self-propelled agents with a tendency to flock. The inter-particle interactions in the CS model are all-to-all and symmetric, akin to many forces typical for classical physics. However, recently it has been recognized that asymetric interactions with finite range are fundamental from the point of view of applications, for instance, in systems related to human and animal behavior. Among CS-type models with non-standard interactions is the so called $q$-closest  model \cite{CD-16} also known as the Cucker-Dong (CD) model, which describes the motion of $N$ particles following the ODE system

\medskip
\begin{equation}\label{qclosed}
\begin{cases}
\displaystyle   \dot x_i=v_i \\[7pt]
\displaystyle  \dot v_i = -\frac1q \sum_{j\in {\mathcal N}_i} (v_i-v_j)
 \end{cases}
\end{equation}

\medskip
\noindent
where $(x_i(t),v_i(t))$ denotes the position and velocity of $i$th particle at the time $t\geq 0$. Set ${\mathcal N}_i =  \mathcal{N}_i(t)$ consists of the indices of $q$-closest neighbors of the $i$th agent at the time $t$. Throughout the paper we identify the indices in $\mathcal{N}_i$ with their corresponding particles.
%
%

 Our goal is to tackle the issue of ill-posedness of \eqref{qclosed} related to the problems with the rigorous definition of ${\mathcal N}_i$. We introduce a fuzzy 
 q-closest model which is well-posed. Moreover we propose a kinetic variant of the said model and prove a stability estimate for its measure-valued solutions in the $2$-Wasserstein metric. Finally, we use the stability estimate to derive solutions to the kinetic equation as a mean-field limit of the ODE system. We end our considerations with propositions of further developments related to the mean-field limit and asymptotics for the original CD model. The latter point is left unanswered, we have conjectures for large time behavior, but still without the complete proof for $q\leq \frac{N}{2}$.

\medskip

The main feature of system (\ref{qclosed}) is its combinatorical character.
The interactions depend on the inter-particle distance only as far as definition of the neighbor sets $\N_i$ is concerned. This is different to the typical CS and CD setting, where the so-called communication weight is involved; basically each summant on the right-hand side of the second equation of \eqref{qclosed} is multiplied by a function of the inter-particle distance $\psi(|x_i-x_j|)$. The main purpose of the communication weight $\psi$ is to ensure that the interactions decay with the distance, i.e. $\psi(|x_i-x_j|)\to 0$ as $|x_i-x_j|\to+\infty$. From the perspective of the CS model, one can differentiate between two types of communication: regular (say bounded and Lipschitz continuous) and singular (e.g. $\psi(r)=r^{-\alpha}$, $\alpha>0$). The case of regular communication weight has been thoroughly investigated over past years leading to results in asymptotics and large-time behavior \cite{HL-09, PKH-10, PGE-09}, influence of additional forces \cite{BH-17, BCC-11, HHK-10} and control \cite{AKMO-20, BBCK-18, CKPP-19}. Other directions include stability and mean-field limit \cite{HL-09, HT-08}. The case of the singular communication weight is more delicate resulting in interesting behaviors such as sticking of the trajectories of the particles \cite{P-14}, or for $\alpha\geq 1$ complete collision-avoidance \cite{CCMP-17, M-18, YYC-20}. The issue of well-posedness is also delicate; the weakly singular case with $\alpha\in(0,1)$ has beed studied in \cite{P-15} followed by \cite{MP-18} and more recently in \cite{CZ-21} and \cite{PP-22-arxiv}, while the strongly singular case remains completely open on the kinetic level (well posedness for the ODE was proved in \cite{CCMP-17}). For more information on the CS models and its variants for both regular and singular communication weight we recommend the surveys \cite{CHL-17, MMPZ-19}.

Here, we assume that the communication weight is simply $\psi\equiv 1$, which together with the influence restricted to the neighbor sets introduces an interesting combinatorical challenge. From the perspective of well-posedness regular communication does not introduce any more difficulty than $\psi\equiv 1$, albeit computations would be somewhat more complicated. Singular communication weights, naturally introduce problems, which are mostly unrelated to the essence of the CD model -- the rigorous treatment of the neighbor sets. We decided to focus solely on the problem of well-posedness related to the neighbor sets.

Interestingly, without the decay introduced by the communication weight, the interaction between the particles are more chaotic, effectively switching between $0$ (off) and $1$ (on) depending on how the particles are situated in relation to each other. Such a perspective is uncommon for the standard analysis where almost all quantities are depending on the distance, they decrease and finally vanish for large distances. We elaborate on this issue at the end of the paper in Section \ref{sec:conc}.



\medskip
\subsection{Critical issue of the rigorous definition of ${\mathcal N}_i$} Let us discuss the neighbor set $\N_i$. Its mathematically rigorous definition is not straightforward. The situation is clear if no two agents are situated at the same distance from $x_i$. Then the definition of ${\mathcal N}_i$ can be stated as follows:

\medskip
$j\in {\mathcal N}_i$ if $i\neq j$ and 
\begin{equation}\label{Ni}
    \sharp \{ x_k: {\rm dist}(x_i,x_k) < {\rm dist}(x_i,x_j)\} <q.
\end{equation}

\medskip
\noindent
Here $\sharp  {\mathcal N}_i$ is the cardinality of  $ {\mathcal N}_i$. Then of course we find that $\sharp {\mathcal N}_i=q$. 
However the general case is more complex and the above definition does not fit to all cases. Problems occur when exactly $q_0<q$ particles satisfy the sharp equation in \eqref{Ni} and there are $q_1> q-q_0$ particles at the next smallest admissible distance from $x_i$ (with the same distance from $x_i$). Thus, in order to have  $\sharp {\mathcal N}_i=q$, we need to somehow choose $q-q_0$ particles from $q_1$ particles. This is impossible without additional rules. There are a couple of ways to circumvent the above issue.

\smallskip 
\begin{itemize}
    \item Cucker and Dong \cite{CD-16} prescribe an arbitrary order to the particles and define the set $\mathcal{N}_i$ by (\ref{Ni}) provided that $\sharp \mathcal{N}_i=q$ and if it $\sharp \mathcal{N}_i> q$ the definition takes $q-q_0$ particles with the smallest indices from the the group of $q_1$ particles with the same distance from $x_i$.
    Then  $\sharp {\mathcal N}_i =q$ and $\N_i$ is well defined.
    \item Another option would be to simply take all $q_1$ particles which are at the same distance from $x_i$ and accept that at times $\sharp {\mathcal N}_i>q$. Then replace the factor $\frac{1}{q}$ in \eqref{qclosed} with $\frac{1}{\sharp {\mathcal N}_i}$, which then is ''usually'' (but not always) equal to $q$.
\end{itemize}
\smallskip 

\noindent
Regardless of the chosen strategy we can easily show that the system is ill posed from the viewpoint of stability--uniqueness properties, see Example \ref{ex1} below.
Our goal is to modify the definition of ${\mathcal N}_i$ so that it retains its original features, while simultaneously allowing well-posedness of the system. To this end we introduce the fuzzy $q$-closest alignment system.

\medskip

Fix $\sigma>0$ and a function $\phi:\R^d \to \R$, such that $\phi \geq 0$ and 
\begin{equation}\label{phi}
{\rm supp}\, \phi \subset 
B(0,\sigma) \mbox{ \ \  and  \ \ } \int_{\R^d} \phi(x)dx  =1.
\end{equation}
The fuzzy variant of \eqref{qclosed}, supplemented by initial data, then reads

\begin{equation}\label{fuzzy}
    \begin{cases}
    \displaystyle \dot{x}_i = v_i, & (x_i(0),v_i(0))=(x_{0i},v_{0i}),\\
    \displaystyle\dot v_i =-\frac{1}{q} \sum_{j=1}^N  \int_{B(x_i,R)} \phi(z-x_j) (v_i-v_j)\, \dd z,
    \end{cases}
\end{equation}

\smallskip
\noindent
where $R=R(x_i(t))$ is the unique radius such that
\begin{align*}
    \sum_{j=1}^N\int_{B(x_i,R(x_i(t)))}\phi(z-x_j)\, \dd z = q.
\end{align*}

\smallskip
\noindent
Intuitively, the particles' perception is fuzzy -- if a neighbor is situated at the position $x_i$, it is perceived to be spread in an area around $x_i$ with radius $\sigma$. Such a phenomenon seems realistic as noise and other inadequacies related to perception and communication occur routinely in real life.  Observe that for any $0\leq q\leq N$ existence and uniqueness of $R$ follows from smoothness of $\phi$. We emphasise that this definition is valid for all real $q\in[0,N]$. Thanks to this observation system \eqref{fuzzy} can be restated in a more convenient and more general way as a kinetic equation, where the ODE concept of particles looses its meaning. First we fix any

$$\eta>0,\qquad q=\eta  N$$

\smallskip
\noindent
which serves as the proportion of agents that should be included in the neighbor sets. Then at each $t\geq 0$ let $f=f_t(x,v)$ be the probabilistic distribution of particles moving through position $x$ with velocity $v$. We define the kinetic, fuzzy $q$-closest alignment equation as follows
\begin{align}\label{mezo1}
\begin{aligned}
    \partial_t f_t + v\cdot\nabla f_t + {\rm div}_v(F[f_t]f_t) &= 0,\\
    F[f_t](x,v) &= \frac{1}{\eta}\intr  \int_{B(x,R[f_t](x))}\phi(z-y)(v-w)\, \dd z \, \dd f_t(y,w),
\end{aligned}
\end{align}

\smallskip
\noindent
where the radius $R$ is defined for any probability measure $\mu$ by 
\begin{equation}\label{defr}
     R[\mu](x) = \inf\Bigg\{r>0: \intr\int_{B(x,r)}\phi(z-y)\, \dd z\, \dd \mu(y,w)\geq \eta\Bigg\}.
\end{equation}

\smallskip
\noindent
Naturally, any solution
$$ t\mapsto (\vx(t),\vv(t)):= \Big((x_1(t),v_1(t)),...,(x_N(t),v_N(t))\Big) $$

\smallskip
\noindent
to \eqref{fuzzy} can be identified with the measure-valued atomic function
\begin{equation}\label{atom}
    t\mapsto f^N_t(x,v) = \frac{1}{N}\sum_{i=1}^N \delta_{x_i(t)}(x)\otimes\delta_{v_i(t)}(v)
\end{equation}

\smallskip
\noindent
solving \eqref{mezo1} with $R[f^N_t](x_i)=R(x_i)$, see Proposition \ref{chain} in the appendix.

\subsection{Main results}

The starting point of our analysis of system \eqref{fuzzy} is the above observation that solutions to \eqref{fuzzy} are equivalent to atomic (cf. \eqref{atom}) measure-valued (or simply distributional) solutions to \eqref{mezo1}, at least provided that \eqref{mezo1} is well-posed. Thus, we perform the analysis of well-posedness of the microscopic system \eqref{fuzzy} on the more robust level of well-posedness of \eqref{mezo1} in the sense of measure-valued solutions. The main results are expressed in the following two theorems related to well-posedness and stability of \eqref{mezo1}, respectively. We denote the space of probability measures with finite second moment over $\R^{2d}$ by ${\mathcal P}_2(\R^{2d})$.

\smallskip
\begin{theo}[Well-posedness in the kinetic case]\label{main}
Let $T>0$ and $f_0\in {\mathcal P}_2(\R^{2d})$  with 
\begin{align*}
    {\mathcal D}_x := {\rm diam}_x ({\rm supp}f_0) <\infty,\quad {\mathcal D}_v := {\rm diam}_v ({\rm supp})f_0 < \infty
\end{align*}

\smallskip
\noindent
be given. Then there exists a unique measure-valued solution to \eqref{mezo1} in the sense of Definition \ref{weaksol} (see below) issued at $f_0$. Furthermore, it is recovered as a mean-field limit of atomic solutions of the form \eqref{atom}, which follow trajectories $t\mapsto (\vx(t),\vv(t))$ solving the ODE system \eqref{fuzzy}.
\end{theo}

\medskip

\begin{theo}[$W_2$-stability]\label{stab}
Let $T>0$ and let $f$ and $\widetilde{f}$ be two measure-valued solutions to \eqref{mezo1} in the sense of Definition \ref{weaksol} issued at the initial data $f_0, \widetilde{f}_0\in {\mathcal P}_2(\R^{2d})$ with
\begin{align*}
    {\mathcal D}_x := {\rm diam}_x ({\rm supp}f_0) <\infty,\quad {\mathcal D}_v := {\rm diam}_v ({\rm supp})f_0 <\infty,\\
    \widetilde{{\mathcal D}}_x := {\rm diam}_x ({\rm supp}\widetilde{f}_0) <\infty,\quad \widetilde{{\mathcal D}}_v := {\rm diam}_v ({\rm supp})\widetilde{f}_0 <\infty,
\end{align*}

\smallskip
\noindent
Then the following stability estimate holds
\begin{equation}\label{stab-1}
   \sup_{t\in[0,T]} W_2(f_t,\widetilde{f}_t)\leq e^{(2c_2+c_1^2e^{2c_2 T}) t} W_2^2(f_0,\widetilde{f}_0),
\end{equation}

\smallskip
\noindent
where $W_2$ is the $2$-Wasserstein metric, while $c_1$ and $c_2$ are positive constants depending on $\phi, d, \eta$ and the diameters of supports of the solutions ${\mathcal D}_x, {\mathcal D}_v, \widetilde{{\mathcal D}}_x, \widetilde{{\mathcal D}}_v$. 
\end{theo}

\begin{rem}\label{remarkpromo}\rm
It is worthwhile to note that the choice of the metric in Theorem \ref{stab} is quite arbitrary. Indeed, based on the assumptions of Theorems \ref{main} and \ref{stab}, and on Definition \ref{weaksol}, we are dealing with solutions that have uniformly bounded supports (controlled by the initial data). Thus all of the associated families of measures are uniformly tight and have all moments uniformly bounded. In such a scenario all of the standard Wasserstein metrices are equivalent to the narrow topology, wherein convergence is characterised by testing the measure with all bounded-continuous functions (in fact in our case, all of the above topologies coincide with the weak$-*$ topology on the space $C_c$ of all continuous, compactly supported functions). Our choice to use the $W_2$ framework is mostly arbitrary, but it has certain advantages related to the fact that models of second-order collective dynamics have close relation to gradient-flows with respect to $W_2$. As discovered recently in \cite{PP-22-1-arxiv} and \cite{PP-22-arxiv} CS-type interactions can be equivalently described as gradient-flows with respect to the so-called fibered Wasserstein distance, which requires a weakly Riemannian structure similar to that of the  classical work \cite{O-01}, by {\sc F. Otto} which specifically requires the $W_2$ topology.
\end{rem}

The literature on the $q$ closest  model concerns mostly asymptotics, with the most prominent result obtained by none other than Cucker and Dong in \cite{CD-16}. It basically says that if $q>\frac{n}{2}$ then the ensemble is connected and flocking is likely to occur (depending on the tail of communication weight $\psi$). This result has been extended to other variants of the models such as the time-delayed CS model \cite{DHJK-20}. However the story of the CS model with non-standard and asymetric communication does not end (nor does it begin) here. One should mention the Motsch-Tadmor model \cite{MTX} that scales the communication with respect to the relative inter-particle distances, and CS model with short-range interactions \cite{J-18, MPT-19}. Another direction are communication protocols, which involve topology that depends on the distribution of the particles, such as the
topological CS model due to Shvydkoy and Tadmor \cite{ST-top-20} or the alignment model operating using the so called density induced consensus protocol \cite{MMP-20, MP-22}. In some sense many of the techniques and models ultimately circle back to the basic question of linear consensus \cite{OLF1, OLF2}, such as in the case of the CS model with switching topologies \cite{DHJK-202}.

What distinguishes our work from the literature is that, to the best of our knowledge, the issue of well-posedness for the CD model has not been solved yet. We propose a way to circumvent the main problems with well-posedness of the particle system and we follow through with a relatively simple derivation of the kinetic $q$ closest model by a mean-field argument. It is performed  by developing a stability estimate in the $2$-Wasserstein topology.

The reminder of the paper is organised as follows. In Section \ref{sec:prelim} we present the preliminaries: example of ill-posedness and lack of large-time stability for \eqref{qclosed} as well as the weak formulation for \eqref{mezo1} and the proof of well-posedness for the ODE system \eqref{fuzzy}. Section \ref{sec:main} is dedicated to the proofs of the main results and finally in Section \ref{sec:conc} we present conclusions and outlook.

\medskip
{\bf Notation.} Throughout the paper we often denote the time-dependence of a measure-valued function $f=f(t,x,v)$ as $f=f_t(x,v)$. For a Borel-measurable mapping $T$ and a probability measure $\mu$, we define the {\it pushforward measure} of $\mu$ along $T$ as $T_\#\mu(A):= \mu(T^{-1}(A))$. The crucial property of the pushforward measure is the change of variable formula, that we use throughout the paper
\begin{equation}\label{change-of-var}
    \int_{A} g \, \dd T_\#\mu = \int_{T^{-1}(A)} g\circ T \, \dd\mu
\end{equation}
assuming that the above integrals are well-defined.

\smallskip
Since we work with measures with uniformly compact supports weak$-*$ convergence is equivalent to narrow convergence (tested by bounded continuous functions) and to convergence in any $p$-Wasserstein metric ($1\leq p<+\infty$). We chose to use the $2$-Wasserstein metric $W_2$ (cf. Remark \ref{remarkpromo}). More general $W_p$ metric, $p\in[1,+\infty]$, is defined on the space of (uniformly) compactly supported probability measures as
\begin{align*}
    W_p(\mu,\sigma) = \left(\inf_{\gamma}\int_{\R^{4d}}|x-x'|^p\, \dd \gamma(x,x')\right)^\frac{1}{p},
\end{align*}

\smallskip

\noindent
where $\gamma$ are all {\it admissible transference plans} between $\mu$ and $\sigma$, i.e. all probabilistic measures on $\R^{4d}$ with $x$-marginal $\pi^x_\#\gamma$
equal to $\mu$ and $x'$-marginal $\pi^{x'}_\#\gamma$ equal to $\sigma$, or equivalently 
$$ \gamma(A\times\R^{2d})=\mu(A),\quad \gamma(\R^{2d}\times B) = \sigma(B). $$
Here we also include two relevant properties of the $W_p$ metric. First, $W_1$ is recovered as a dual to Lipschitz continuous functions, namely
\begin{equation}\label{k-r}
    \left|\int_{\R^{2d}} g(x,v) \dd (\mu - \sigma)\right|\approx [g]_{Lip}W_1(\mu,\sigma),
\end{equation} 
where $[g]_{Lip}$ denotes the Lipschitz constant of the function $g$.
This property is known as the Kantorovitch-Rubinstein duality theorem \cite[Theorem 4.1]{E-12}. Second, by the H\" older inequality
$$ \int_{\R^{4d}}|x-x'|\, \dd \gamma(x,x') \leq \sqrt{\int_{\R^{4d}}|x-x'|^2\, \dd \gamma(x,x')},$$
which ensures that
\begin{equation}\label{w12}
    W_1(\mu,\sigma)\leq W_2(\mu,\sigma). 
\end{equation} 

\smallskip
\noindent
Finally, we shall denote

$$\vx = (x_1,..,x_N)\in \R^{Nd},\quad \vv = (v_1,...,v_N)\in \R^{Nd}.$$

\bigskip

\section{Preliminaries}\label{sec:prelim}
The following section is dedicated to the presentation of various foundational information: in Section \ref{sec:ex} we present an example related to the ill-posedness and instability of \eqref{qclosed}, in Section \ref{sec:weaksol} -- weak formulation for \eqref{mezo1} and an essential, geometrical lemma related to the radius $R$ in \eqref{defr}. At the end of the section we prove the classical well-posedness of \eqref{fuzzy}.

\subsection{Ill-posedness of the $q$-closest alignment system.}\label{sec:ex}~

As stated in the introduction, system \eqref{qclosed} can be ill-posed whenever any of the neighbor sets $\N_i$ for $i\in\{1,...,N\}$ changes in time. The following example shows exactly that, together with the lack of stability.

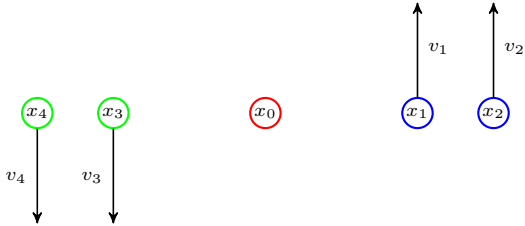
\begin{figure}[h]
\begin{center}

\begin{tikzpicture}[->,>=stealth',shorten >=1pt,auto,node distance=1.5cm,semithick]
\tikzstyle{every state}=[fill=none, draw=none, text=black]

\node (A1) at (-1,0) {\tiny{$x_{4}$}};
\node (A) at (0,0) {\tiny{$x_{3}$}};
\node (B) at (2,0) {\tiny{$x_0$}};
\node (C) at (4,0) {\tiny{$x_1$}};
\node (C1) at (5,0) {\tiny{$x_2$}};

\draw[red,thick] (2,0) circle (1/5);
\draw[blue,thick] (4,0) circle (1/5);
\draw[blue,thick] (5,0) circle (1/5);
\draw[green,thick] (0,0) circle (1/5);
\draw[green,thick] (-1,0) circle (1/5);

\path (A) edge node [left] {\tiny{$v_{3}$}} (0,-1.5);
\path (A1) edge node [left] {\tiny{$v_{4}$}} (-1,-1.5);
\path (C) edge node  [right] {\tiny{$v_1$}} (4,1.5);
\path (C1) edge node  [right] {\tiny{$v_2$}} (5,1.5);
\end{tikzpicture} \end{center} 

\caption{Initial configuration of the ensemble in {\bf Example 1.}}
		\label{fig:circle}
\end{figure}

\smallskip
\begin{example}\rm\label{ex1}
Consider the $q$-closest alignment system \eqref{qclosed} in the case of $N=5$ and $q=2$, and the following initial configuration on the plane
\begin{equation}\label{conf}
    \begin{array}{ll}
    x_0=(0,0)     &  v_0=(0,0)\\
    x_{1}=(1,0)     & v_{1}=(0,1) \\
    x_{2}=(1.1,0) & v_{2}=(0,1) \\
    x_{3}=(-1,0)     & v_{3}=(0,-1) \\
    x_{4}=(-1.1,0) & v_{4}=(0,-1).
    \end{array}
\end{equation}

\smallskip
\noindent
Observe that such initial configuration leads to the lack of uniqueness. Particles with indices $1,2$ and with indices $3,4$ constitute two stable clusters travelling with the speed $(0,1)$ and $(0,-1)$, respectively. The issue is the interaction with $0$th particle. Note that we have the following solutions:
$$
v_0^1(t)=1-e^{-1/2 t} \mbox{ \ \ and then \ \ } x_0^1(t)=t -\int_0^t e^{-1/2 s} ds.
$$

\smallskip
\noindent
and 
$$
v_0^2(t)=e^{-1/2 t}-1 \mbox{ \ \ and then \ \ } x_0^2(t)=-t +\int_0^t e^{-1/2 s} ds.
$$

\smallskip
\noindent
Both solutions start from the same configuration, and both satisfy \eqref{qclosed}. Of course this example showcases also instability of configuration \eqref{conf}. 
One can consider the $0$th agent of the form
\begin{equation*}
    x_0=(0,\epsilon), \qquad v_0=(0,0)
\end{equation*}

\smallskip
\noindent
Then we pass with $\epsilon$ to $0$, and observe that
\begin{equation*}
\left|    \lim_{\epsilon \to 0^+} x_0(1) -  \lim_{\epsilon \to 0^-} x_0(1)\right| > 1/10.
\end{equation*}

\smallskip
\noindent
Similar problems arise if we replace the strict inequalities in \eqref{Ni} with their non-strict variant. Note all the above can be easily generalized on the case $q>2$.
\end{example}

Interestingly, the same example as above but on the 2D torus showcases a situation, when no asymptotic alignment nor flocking occurs (even within separate clusters). Indeed, in Example \ref{ex1} on a torus, one can see that the particle $0$ would indefinitely switch between being influenced by the cluster $\{1,2\}$ and by $\{3,4\}$. Consequently $v_0(t)$ never stabilizes. We explore this direction more at the end of the paper in Section \ref{sec:conc}.




\subsection{Weak formulation for the kinetic $q$-closest system}\label{sec:weaksol}

Next, we provide the weak formulation for the kinetic equation \eqref{mezo1}.

\begin{defi}\label{weaksol}
Let $T>0$ and $f_0\in {\mathcal P}_2(\R^{2d})$ be compactly supported in $\R^{2d}$ with
\begin{align*}
    {\mathcal D}_x:= {\rm diam}_x({\rm supp}f_0) <+\infty,\qquad  {\mathcal D}_v:= {\rm diam}_v({\rm supp}f_0) <+\infty.
\end{align*}

\smallskip
\noindent
We say that $f$ is a measure-valued solution to \eqref{mezo1} issued in the initial datum $f_0$ on the time interval $[0,T]$ if
\begin{enumerate}[(i)]
    \item The function $t\mapsto f_t$ belongs to the space $C([0,T]; ({\mathcal P}_2(\R^{2d}),W_2))$.
    
    \smallskip
    \item The function $f_t$ is compactly supported uniformly with respect to $t\in[0,T]$ with
    \begin{align*}
        {\mathcal D}_x(t):={\rm diam}_x ({\rm supp}f_t) \leq {\mathcal D}_x + t{\mathcal D}_v,\qquad {\mathcal D}_v(t):={\rm diam}_v ({\rm supp}f_t)\leq {\mathcal D}_v.
    \end{align*}
    
    \smallskip
    \item For any function $\psi\in C^1_b([0,T]\times \R^{2d})$ compactly supported in $[0,T)$ we have
    \begin{equation}\label{weaksol-1}
        \begin{split}
        \int_0^T\intr\Big( \partial_t\psi(t,x,v) + v\cdot \nabla_x\psi(t,x,v) + \nabla_v\psi(t,x,v)\cdot F[f_t](x,v)\Big)\, \dd f_t(x,v)\,\dd t \\
        = -\intr \psi(0,x,v) \,\dd f_0(x,v).
        \end{split}
    \end{equation}
\end{enumerate}

\end{defi}

\medskip

\begin{rem}\label{rem-convol}\rm
Hereinafter we shall use an alternative representation of $F[f_t]$ in \eqref{mezo1} and of the integral in the definition of $R$ in \eqref{defr}. Namely, by Fubini's theorem, we have
\begin{align*}
    F[f_t](x,v) &= \frac{1}{\eta}\int_{B(x,R[f_t](x))} (\phi * m_1f_t)(z)\, \dd z,\\
     R[f_t](x) &= \inf\Bigg\{r>0: \int_{B(x,r)}(\phi* m_0f_t)(z)\, \dd z\geq \eta\Bigg\},
\end{align*}

\smallskip
\noindent
where 
\begin{equation}\label{m1-m0}
m_1f_t(x) := \int_{\R^d} v\, \dd f_t(x,v)\quad \mbox{and}\quad m_0f_t(x) := \int_{\R^d}\, \dd f_t(x,v)
\end{equation}
are the first and second local velocity moments of $f_t$, respectively.
\end{rem}

\medskip

One of the main tools used throughout the paper is the observation that for any $\mu$ the radius $R[\mu]$ is Lipschitz continuous with constant 1. As a matter of fact this is a geometrical property completely independent of the mollification $\phi$.


\begin{lem}[Lipschitz continuity of $R$]\label{lem:R}
Let $\mu$ be a positive Radon measure defined over $\R^d$, then $R[\mu]$ defined by \eqref{defr} is Lipschitz continuous with constant 1, i.e.
\begin{equation}\label{Lip-R}
    |R[\mu](x)-R[\mu](x')|\leq |x-x'|.
\end{equation}
Furthermore if $\mu$ is compactly supported then $R[\mu]$ is bounded by ${\rm diam} ({\rm supp}\mu)$, i.e.
\begin{equation}\label{R-bound}
    |R[\mu](x)|\leq {\rm diam} ({\rm supp}\mu)+\sigma \quad \mbox{for }\mu-{a.a. }\ x,
\end{equation}

\smallskip
\noindent
where we recall that $\sigma>0$ is the fixed range of the support of the mollification $\phi$ (cf. \eqref{phi}).
\end{lem}

\medskip

\begin{proof}
 Since $\mu$ is fixed we denote $R(x):=R[\mu](x)$. Take any $x$ and $x'$, and assume without loss of generality that $R(x')\leq R(x)$. Note that we always have
\begin{equation*}
    B(x',R(x')) \subset B(x,R(x')+|x-x'|).
\end{equation*}

\smallskip
\noindent
The above inclusion together with \eqref{defr} for $R(x')$, lead to 
\begin{align}\label{lem:R-1}
    \intr\int_{B(x,R(x')+|x-x'|)}\phi(z-y)\, \dd z\, \dd \mu(y,w)\geq \intr\int_{B(x',R(x'))}\phi(z-y)\, \dd z\, \dd \mu(y,w) \geq \eta,
\end{align}

\smallskip
\noindent
since we integrate a positive function over a larger set. But definition \eqref{defr} for $R(x)$ implies that $R(x)$ is the infimum of all $r$ from \eqref{defr}, thus by 
\eqref{lem:R-1} we deduce 
\begin{equation*}
    R(x) \leq R(x')+|x-x'|.
\end{equation*}

\smallskip
\noindent
Recalling that $R(x')\leq R(x)$ this finishes the proof of \eqref{Lip-R}. Now suppose that $\bar{R}:={\rm diam}({\rm supp}\mu)<+\infty$ and fix $x\in {\rm supp}\mu$ so that ${\rm supp}\mu\subset B(x,\bar{R})$. Using the representation from Remark \ref{rem-convol}, we have
$$ \intr\int_{B(x,\bar{R}+\sigma)}\phi(z-y)\, \dd z\, \dd \mu(y,w) = \int_{B(x,\bar{R}+\sigma)}\phi* m_0\mu(z)\, \dd z, $$

\smallskip
\noindent
which together with the fact that ${\rm supp} (\phi*m_0\mu) = {\rm supp}\phi + {\rm supp}\ m_0\mu\subset B(x,\bar{R}+\sigma)$ implies that

$$ \intr\int_{B(x,\bar{R}+\sigma)}\phi(z-y)\, \dd z\, \dd \mu(y,w) = 1>\eta $$

\smallskip
\noindent
and by \eqref{defr} inequality \eqref{R-bound} necessarily follows.

\end{proof}

\bigskip

\subsection{Well-posedness of the fuzzy $q$-closest system \eqref{fuzzy}}
In this section we prove the well-posedness of the fuzzy system \eqref{fuzzy} together with some of its basic properties. 
 
\begin{prop}\label{th:odes}
Let $(\vx_0,\vv_0) \in \R^d \times \R^d$, $q, T>0$.
Then there exists a unique global in time solution $(\vx,\vv)\in C^1([0,T])^d$ to problem \eqref{fuzzy}, such that
\begin{equation}\label{est-odes}
    \sup_{t\in[0,T]}|\vv(t)|_\infty\leq |\vv_0|_\infty,\qquad \sup_{t\in[0,T]}|\vx(t)|_\infty\leq |\vx_0|_\infty+ T|\vv_0|_\infty.
\end{equation}

\smallskip
\noindent
Furthermore for all $t\in[0,T]$, we have
\begin{equation}\label{vel-conv}
    {\rm conv}(v_1(t),...,v_N(t)) \subset {\rm conv}
    (v_{01},...,v_{0N}).
\end{equation}
Here and hereinafter ${\rm conv} A$ denotes the convex hull of $A$.
\end{prop}

\medskip

\begin{proof}
Thanks to Lemma \ref{lem:R} we are ensured that the right-hand side of the system is Lipschitz continuous in $x$ and $v$, so unique solutions in the $C^1$ class exists at least locally in time.
In order to show the global existence we need to prove global bounds.

\medskip
The form of the right-hand side of system \eqref{fuzzy} implies immediately the bound for the velocity. 
Indeed, let $\mathcal{V}(t)={\rm conv}(v_1(t),...,v_N(t))$, then by definition

\smallskip
\begin{equation*}
    \mbox{ as \ } w_{max} \mbox{ realizes the maximum} \sup_{w \in \partial \mathcal{V}(t)} |w|,  
\end{equation*}
then 
$$
\mbox{ for all } v_k \mbox{ \ with \ } k=1,...,N \qquad
v_k \cdot w_{max} - |w_{max}|^2\leq 0
$$

\smallskip
\noindent
and equality holds only as $w_{max}=v_k$ for some $k$.
Let us multiply \eqref{fuzzy}$_2$ by $v_k=w_{max}$ obtaining
\begin{equation*}
   \frac{1}{2} \frac{d}{dt} |v_k|^2 = -\frac{1}{q} \sum_{l}  \int_{B(x_k,R(x_k))} \phi(z-x_l)(v_k - v_l)\cdot v_k \leq 0,
\end{equation*}

\smallskip
\noindent
which leads to the first part of (\ref{est-odes}), and then to the second assertion.


\medskip
In order to prove more precise information about the evolution of the set of velocities note that since the right-hand side is smooth enough the system can be effectively approximated by the semi-discretization  in time. So for $\Delta t >0$
we consider
\begin{equation}\label{R-2}
    v_k(t+\Delta t)= v_k(t) - {\Delta t} \sum_l \theta_l(v_k-v_l)
\end{equation}

\smallskip
\noindent
with $\theta_l = \frac{1}{q} \int_{B(x_k,R(x_k))} \phi(z-x_l)dz$. It's clear $\theta_l \geq 0$ and 
$\sum_l \theta_l=1$. So (\ref{R-2}) takes the form
\begin{equation*}
    v_k(t+\Delta t)= (1-\Delta t) v_k(t) + {\Delta t} \sum_l \theta_l v_l. 
\end{equation*}

\smallskip
\noindent
Hence $v_k(t+\Delta t)$ is a convex combination of the rest of velocities at the previous time step
(for $\Delta t <1$). So for the semi-discretization we get
\begin{equation*}
    \mathcal{V}(t) \supset \mathcal{V}(t+\Delta t).
\end{equation*}

\smallskip
\noindent
The above relation is preserved for the limit $\Delta t \to 0$. We have (\ref{vel-conv}) and the proof is done.
\end{proof}
\bigskip

\section{Proof of the main results}\label{sec:main}
The following section is dedicated to the proof of well-posedness of system \eqref{mezo1} in the class of solutions delivered by Definition \ref{weaksol} i.e. Theorem \ref{main}. We begin with the second result -- the $W_2$-Wasserstein stability estimate for solutions of \eqref{mezo1} -- Theorem \ref{stab}, which enables us to obtain weak solutions of \eqref{mezo1} as a mean-field limit of atomic solutions originating from the ODE system \eqref{fuzzy} (and grants uniqueness).

\medskip

\subsection{Representation formula}\label{sec:char}

We start with the representation formula for any solution of \eqref{mezo1} in the sense of Definition \ref{weaksol}.

\begin{prop}\label{character}
Let $f$ be a measure-valued solution to \eqref{mezo1} in the time interval $[0,T]$ issued at the compactly supported initial datum $f_0\in {\mathcal P}_2(\R^{2d})$. Then there exists a unique mapping ${\mathcal Z}_t[f](z):=({\mathcal X}_t(x,v), {\mathcal V}_t(x,v))$, $z=(x,v)\in\R^{2d}$, satisfying the ODE system
\begin{equation*}
\begin{cases}
   \displaystyle \frac{d}{dt}{\mathcal X}_t(x,v) = {\mathcal V}_t(x,v),& ({\mathcal X}_0(x,v), {\mathcal V}_0(x,v)) = (x,v),\\[7pt]
    \displaystyle\frac{d}{dt}{\mathcal V}_t(x,v) = F[f_t]({\mathcal X}_t(x,v), {\mathcal V}_t(x,v)).\\
\end{cases}
\end{equation*}

\smallskip
\noindent
Moreover $f_t = ({\mathcal Z}_t)_{\#}f_0$.
\end{prop}

\smallskip
\noindent
\begin{proof}

By Lemma \ref{lip} below, the velocity field $F[f_t]$ is uniformly Lipschitz continuous and bounded for any measure-valued solution $f$ of \eqref{mezo1}. Thus the characteristics for \eqref{mezo1} are well defined and unique. Since \eqref{mezo1} is ultimately a continuity equation, albeit in the combined variable $(x,v)\in\R^{2d}$, the representation formula in Proposition \ref{character} follows directly from the classical theory \cite[Proposition 8.1.8]{AGS-08}. Thus, the issue is the Lipschitz continuity of $F[f_t]$ which we explore in the following lemma.

That said, for the reader's convenience we shall at least provide a formal argument based on the change of variables formula \eqref{change-of-var}, which supports the claim that $f_t = ({\mathcal Z}_t)_{\#}f_0$ satisfies \eqref{mezo1} in the distributional sense. To such an end let $\psi$ be a smooth test function, compactly supported in $[0,T)\times\R^{2d}$. Using the change of variables formula \eqref{change-of-var} and the chain rule we obtain
\begin{align*}
    \int_0^T\intr\Big( \partial_t\psi(t,z) + v\cdot \nabla_x\psi(t,z) + \nabla_v\psi(t,z)\cdot F[f_t](z)\Big)\, \dd ({\mathcal Z}_t)_{\#}f_0(z)\,\dd t \\
    = \int_0^T\intr\Big( (\partial_t\psi)(t,{\mathcal Z}_t(z)) + v\cdot (\nabla_x\psi)(t,{\mathcal Z}_t(z)) + (\nabla_v\psi)(t,{\mathcal Z}_t(z))\cdot F[f_t]({\mathcal Z}_t(z))\Big)\, \dd f_0(z)\,\dd t\\
    = \int_0^T\intr\Big( \frac{d}{dt}\psi(t,{\mathcal Z}_t(z))\Big)\, \dd f_0(z)\,\dd t = -\intr \psi(0,z)\dd f_0(z),
\end{align*}
which is the weak formulation of \eqref{mezo1}, cf. \eqref{weaksol-1}.

\end{proof}

\begin{rem}\rm
Observe that Proposition \ref{character} already provides existence and uniqueness for \eqref{mezo1}. This is not a surprise -- it is a Vlasov-type equation with a Lipschitz continuous right-hand side after all. However are goal lies beyond that, we aim to show a stability estimate and approximate the solutions by the mean-field limit.
\end{rem}

\begin{lem}\label{lip}
Let $f$ and $\widetilde{f}$ be a pair of measure-valued solutions to \eqref{mezo1}. Then $F[f_t]$ is bounded as a function of $(t,x,v)$ and for a.a. $t\in[0,T)$ we have
\begin{align*}
|F[f_t](x,v)-F[f_t](x',v')|&\leq c_0(|x-x'|+ |v-v'|)\\
|F[f_t](x,v)-F[\widetilde{f}_t](x,v)|&\leq c_1 W_2(f_t,\widetilde{f}_t),
\end{align*}

\noindent
where $c_0, c_1$ are constants depending on $\phi, \eta, d, {\mathcal D}_x$ and ${\mathcal D}_v$.
\end{lem}

\medskip
\begin{proof}
First, we note that since $f$ has a uniformly bounded support and $\phi$ is bounded, it is clear that the function $(t,x,v)\mapsto F[f_t](x,v)$ is bounded by its definition in \eqref{mezo1}.

\smallskip

We proceed with the Lipschitz continuity of $F[f_t]$. Since it is linear with respect to $v$, we have
\begin{align*}
    \nabla_v F[f_t](x,v) = \frac{1}{\eta}1_d,
\end{align*}

\smallskip
\noindent
where $1_d=(1,..,1)\in\R^d$, which concludes the proof of Lipschitz continuity with respect to $v$.
For the Lipschitz continuity with respect to $x$ we denote $R:=R[f_t](x)$ and $R':=R[f_t](x')$, and write
\begin{align*}
    F[f_t](x,v)-F[f_t](x',v) \\= \frac{1}{\eta}\intr w\left(\int_{B(x,R)}\phi(z-y)\, \dd z - \int_{B(x',R')}\phi(z-y)\, \dd z\right)\, \dd f_t(y,w)\\
    = \frac{1}{\eta}\intr w\int_{B(0,1)}\left(R^d\phi(Rz+x-y) - R'^d\phi(R'z+x'-y)\right)\, \dd z\, \dd f_t(y,w)\\
    =(R^d-R'^d)\frac{1}{\eta}\intr w\int_{B(x,R)}\phi(z-y)\, \dd z \, \dd f_t(y,w)\\
    + R'^d\frac{1}{\eta}\intr w\int_{B(0,1)}\left(\phi(Rz+x-y) -\phi(R'z+x'-y)\right)\, \dd z\, \dd f_t(y,w)\\
    =: I + II.
\end{align*}

\smallskip
\noindent
Then, by Lemma \ref{lem:R}
\begin{align*}
    |I|\leq d(R^{d-1}+R'^{d-1})D_v|R-R'|\leq 2d D_x^{d-1}D_v|x-x'|
\end{align*}

\smallskip
\noindent
and
\begin{align*}
    |II|\leq R'^d[\phi]_{{\rm Lip}}\frac{1}{\eta}\intr w\int_{B(0,1)} |z|+1 \, \dd z\, \dd f(t,y,w) |x-x'|\\
    \leq 2D_x^d[\phi]_{{\rm Lip}}\frac{1}{\eta}D_v\omega_d |x-x'|,
\end{align*}

\smallskip
\noindent
where $\omega_d$ is the volume of a $d$-dimensional unit ball. Combining the above yields the first desired inequality with $c_0 = \max\{\frac{1}{\eta}, 2dD_x^{d-1}D_v+ \frac{2}{\eta}D^d_x[\phi]_{Lip}D_v\}$.

\medskip
Next, we prove the pointwise Lipschitz continuity of $f\mapsto F[f_t](x,v)$ with respect to the $W_2$ metric. Using representation of $F[f_t]$ from Remark \ref{rem-convol}, we have
\begin{align*}
   {\mathcal L} &:= F[f_t](x,v)-F[\widetilde{f}_t](x,v)\\
   &= \frac{1}{\eta}\int_{B(x,R[f_t](x))} (\phi * m_1f_t)(z)\, \dd z - \frac{1}{\eta}\int_{B(x,R[\widetilde{f}_t](x))} (\phi * m_1\widetilde{f}_t)(z)\, \dd z,
\end{align*}
where we used notation (\ref{m1-m0}),
%
and thus, assuming without a loss of generality that $R[f](t,x)\geq R[\widetilde{f}_t](x)$,

\smallskip
\begin{align*}
    |{\mathcal L}| &\leq \frac{|B(x,R[f_t](x))|}{\eta}|(\phi * m_1f_t) - (\phi * m_1\widetilde{f}_t)|_\infty\\ 
    &\quad + \frac{1}{\eta}\int_{B(x,R[f_t](x))\setminus B(x,R[\widetilde{f}_t](x))} (\phi * m_1\widetilde{f}_t)(z)\, \dd z
    =: III + IV.
\end{align*}

\smallskip
\noindent
By Kantorovitch-Rubinstein duality theorem for any $z\in\R^d$
\begin{align*}
    |(\phi * m_1f_t)(z) - (\phi * m_1\widetilde{f}_t)(z)| = \left|\int_{\R^{2d}} \phi(z-y)w \, (\dd f_t(y,w) - \dd \widetilde{f}_t(y,w))\right|\\
    \stackrel{\eqref{k-r}}{\leq} L_{\phi,w} W_1(f_t,\widetilde{f}_t) \stackrel{\eqref{w12}}{\leq} L_{\phi,w} W_2(f_t,\widetilde{f}_t),
\end{align*}

\smallskip
\noindent
where $L_{\phi,w} = ([\phi]_{Lip}{\mathcal D}_v + |\phi|_\infty)$ is the Lipschitz constant of the function $$(y,w)\mapsto \phi(z-y)w.$$

\smallskip
\noindent
Thus
\begin{align*}
    III\leq \frac{\omega_d {\mathcal D}_x^d L_{\phi,w}}{\eta}W_2(f_t,\widetilde{f}_t).
\end{align*}

\smallskip
\noindent
On the other hand (recalling (\ref{m1-m0}))
\begin{align*}
    IV\leq \frac{{\mathcal D}_v}{\eta}\int_{B(x,R[f_t](x))\setminus B(x,R[\widetilde{f}_t](x))} (\phi * m_0\widetilde{f}_t)(z)\, \dd z
\end{align*}

\smallskip
\noindent
and we use the relation
\begin{align*}
    \eta = \int_{B(x,R[f_t](x))} (\phi * m_0f_t)(z)\, \dd z = \int_{B(x,R[\widetilde{f}_t](x))} (\phi * m_0\widetilde{f}_t)(z)\, \dd z
\end{align*}

\smallskip
\noindent
to infer that
\begin{align*}
    \int_{B(x,R[f_t](x))\setminus B(x,R[\widetilde{f}_t](x))} (\phi * m_0\widetilde{f}_t)(z)\, \dd z = \int_{B(x,R[f_t](x))} (\phi * m_0f_t)(z) - (\phi * m_0\widetilde{f}_t)(z) \, \dd z.
\end{align*}

\smallskip
\noindent
Using Kantorovitch-Rubinstein duality theorem in a similar manner to the estimation of $III$ we conclude that
\begin{align*}
    IV\leq \frac{\omega_d {\mathcal D}_x^d {\mathcal D}_v}{\eta} W_2(f_t,\widetilde{f}_t).
\end{align*}

\end{proof}

\medskip

\subsection{Stability estimate}\label{sec:stab}
Our next goal is to establish the stability estimate for solutions of \eqref{mezo1} and by doing so -- prove Theorem \ref{stab}. We begin with a lemma, which establishes control over characteristics from the representation formula in Proposition \ref{character}.

\begin{lem}\label{charstab}
Let $f$ and $\widetilde{f}$ be a pair of measure-valued solutions to \eqref{mezo1} issued at initial data $f_0$ and $\widetilde{f_0}$. Assume further that ${\mathcal Z}_t[f]$ and ${\mathcal Z}_t[\widetilde{f}]$ are characteristics defined in Proposition \ref{character} for $f$ and $\widetilde{f}$, respectively. Then there exists a positive constant $c_2=c_2(c_0)$, such that

\smallskip
\begin{align}
    |\mathcal{Z}_t[f](z)-\mathcal{Z}_t[f](z')|^2 &\leq e^{2c_2t}|z-z'|^2,\label{charstab-1}\\
    |\mathcal{Z}_t[f](z)- \mathcal{Z}_t[\widetilde{f}](z)|^2 &\leq c_1^2e^{2c_2 t}\int_0^t W_2^2(f_t,\widetilde{f}_t)\, \dd t.\label{charstab-2} 
\end{align}
\end{lem}

\medskip

\begin{proof}
Fix $z$ and $z'$ in $\R^{2d}$. We have
\begin{align*}
    \frac{d}{dt}|\mathcal{Z}_t[f](z)-\mathcal{Z}_t[f](z')|^2 &= 2(\mathcal{X}_t[f](z)-\mathcal{X}_t[f](z'), \mathcal{V}_t[f](z)-\mathcal{V}_t[f](z'))\\
   &\quad \cdot(\mathcal{V}_t[f](z)-\mathcal{V}_t[f](z'), F[f_t]({\mathcal Z}[f](z))-F[f_t]({\mathcal Z}[f](z')))\\
    &\leq (2+2c_0^2)|\mathcal{Z}_t[f](z)-\mathcal{Z}_t[f](z')|^2,
\end{align*}

\smallskip
\noindent
where we use Lemma \ref{lip} and Young's inequality in the last line. By Gronwall's inequality we recover \eqref{charstab-1} with $c_2 = 1+c_0^2$.

\medskip

To prove inequality \eqref{charstab-2} we proceed similarly. Fix $z\in\R^{2d}$ and compute using Young's inequality
\begin{align*}
     \frac{d}{dt}|\mathcal{Z}_t[f](z)-\mathcal{Z}_t[\widetilde{f}](z)|^2 = 2(\mathcal{X}_t[f](z)-\mathcal{X}_t[\widetilde{f}](z), \mathcal{V}_t[f](z)-\mathcal{V}_t[\widetilde{f}](z))\\
    \cdot (\mathcal{V}_t[f](z)-\mathcal{V}_t[\widetilde{f}](z), F[f_t]({\mathcal Z}[f](z))-F[\widetilde{f}_t]({\mathcal Z}[\widetilde{f}](z)))\\
    \leq 2|\mathcal{Z}_t[f](z)-\mathcal{Z}_t[\widetilde{f}](z)|^2 + 2|F[f_t]({\mathcal Z}[f](z))-F[\widetilde{f}_t]({\mathcal Z}[f](z))|^2\\
    + 2|F[\widetilde{f}_t]({\mathcal Z}[f](z))-F[\widetilde{f}_t]({\mathcal Z}[\widetilde{f}](z))|^2.
\end{align*}
In the above inequality, we use Lemma \ref{lip} to estimate
\begin{align*}
    |F[f_t]({\mathcal Z}[f](z))-F[\widetilde{f}_t]({\mathcal Z}[f](z))|^2\leq c_1^2W_2^2(f_t,\widetilde{f}_t)
\end{align*}

\smallskip
\noindent
and
\begin{align*}
    |F[\widetilde{f}_t]({\mathcal Z}[f](z))-F[\widetilde{f}_t]({\mathcal Z}[\widetilde{f}](z))|^2\leq 2c_0^2|\mathcal{Z}_t[f](z)-\mathcal{Z}_t[\widetilde{f}](z)|^2,
\end{align*}

\smallskip
\noindent
yielding
\begin{align*}
     \frac{d}{dt}|\mathcal{Z}_t[f](z)-\mathcal{Z}_t[\widetilde{f}](z)|^2\leq (2+2c_0^2)|\mathcal{Z}_t[f](z)-\mathcal{Z}_t[\widetilde{f}](z)|^2 + c_1^2W_2^2(f_t,\widetilde{f}_t),
\end{align*}

\smallskip
\noindent
which, by Gronwall's inequality leads to \eqref{charstab-2}.
\end{proof}

With Lemma \ref{charstab} at hand we are finally ready to prove Theorem \ref{stab}.

\begin{proof}[Proof of Theorem \ref{stab}]

Let $f$ and ${\widetilde f}$ be any two solutions of \eqref{mezo1} with initial data $f_0$ and ${\widetilde f}_0$ respectively. Furthermore let $\gamma_0$ be an optimal transference plan between $f_0$ and ${\widetilde f}_0$. Then $\gamma = ({\mathcal Z}_t[f],{\mathcal Z}_t[\widetilde{f}])_\#\gamma_0$ is an admissible transference plan between $f_t = ({\mathcal Z}_t[f])_\# f_0$ and ${\widetilde f}_t = ({\mathcal Z}_t[\widetilde{f}])_\#\widetilde{f}_0$, cf. Proposition \ref{character}. Since $\gamma$ is an admissible plan, by the definition of $W_2$ metric and by the change of variables formula \eqref{change-of-var}, we have
\begin{align*}
    W_2^2(f_t,\widetilde{f}_t) \leq \int_{\R^{4d}}|z-z'|^2\, \dd\gamma(z,z') = \int_{\R^{4d}}|{\mathcal Z}_t[f](z_0)-{\mathcal Z}_t[\widetilde{f}](z_0')|^2\, \dd\gamma_0(z_0,z_0')=:{\mathcal L}.
\end{align*}    
Then we write using Lemma \ref{charstab}
\begin{align*}
{\mathcal L}
&\lesssim \int_{\R^{4d}} |{\mathcal Z}_t[f](z_0)-{\mathcal Z}_t[f](z_0')|^2\, \dd\gamma_0(z_0,z_0')\\
&\qquad+ \int_{\R^{4d}}|{\mathcal Z}_t[f](z_0')-{\mathcal Z}_t[\widetilde{f}](z_0')|^2\, \dd\gamma_0(z_0,z_0')\\
 &   \lesssim e^{2c_2 t}\int_{\R^{4d}}|z_0-z_0'|^2\, \dd\gamma_0(z_0,z_0') + c_1^2e^{2c_2t}\int_0^t W_2^2(f_t,\widetilde{f}_t)\, \dd t\\
  &  = e^{2c_2t}W_2^2(f_0,\widetilde{f}_0) + c_1^2e^{2c_2t}\int_0^t W_2^2(f_t,\widetilde{f}_t)\, \dd t.
\end{align*}

\smallskip
\noindent
Using again Gronwall lemma (integral variant) and simplifying (at the cost of worsening the constants) we end up with \eqref{stab-1}.

\end{proof}

\medskip
\subsection{Mean-field approximation and existence revisited}
In this section we will prove the main Theorem \ref{main}. The proof of uniqueness follows directly from the stability Theorem \ref{stab}. Existence will be achieved by a mean-field limit argument. To this end
consider any compactly supported initial datum $f_0\in {\mathcal M}$. By the law of large numbers there exists a sequence of empirical measures
\begin{align}\label{atomicinit}
\begin{aligned}
    f^N_0(x,v) := \frac{1}{N}\sum_{i=1}^N\delta_{x_{i0}^N}(x)\otimes\delta_{v_{i0}^N}(v),\\
    (x^N_{i0},v^N_{i,0})\in {\rm supp}(f_0) \mbox{ and } x^N_{i0}\neq x^N_{j0} \mbox{ for all } i,j\in\{1,...,N\},
\end{aligned}
\end{align}

\smallskip
\noindent
such that
\begin{align*}
    \lim_{N\to\infty} W_2(f^N_0,f_0) = 0.
\end{align*}

\smallskip
\noindent
Then points $(\vx^N_0, \vv^N_0)$ constitute  initial data for the ODE system \eqref{fuzzy}. Let $(\vx^N, \vv^N)$ be the corresponding solution, which exists by Theorem \ref{th:odes}. Then, by Proposition \ref{chain}, for all $N\in{\mathbb N}$ the family of empirical measures
\begin{align}\label{atomic}
    f^N_t(x,v) = \frac{1}{N}\sum_{i=1}^N\delta_{x_i^N(t)}(x)\otimes\delta_{v_i^N(t)}(v)
\end{align}
satisfies \eqref{mezo1} in the sense of distributions.

\medskip

\begin{proof}[Proof of Theorem \ref{main}]
With the preliminary results of Sections \ref{sec:char} and \ref{sec:stab} at hand we argue as follows. First we note that for our initial $f_0$ there exists an approximation by a sequence of empirical measures $f^N_0$ of the form \eqref{atomicinit}. These serve as initial data $(\vx^N_0,\vv_0^N)$ 
    for the particle system \eqref{fuzzy} as explained at the beginning of this section. The initial data are bounded by
    \begin{align*}
        \sup_{i\in\{1,...,N\}}|x_{i0}| \leq {\mathcal D}_x,\qquad \sup_{i\in\{1,...,N\}}|v_{i0}| \leq {\mathcal D}_v. 
    \end{align*}

\smallskip
\noindent
Furthermore, by Theorem \ref{th:odes} for each $N\in{\mathbb N}$ the particle system admits a unique classical solution $(\vx^N,\vv^N)$ in $[0,T]$ which is bounded by
    \begin{align*}
        \sup_{i\in\{1,...,N\}}|x_i|_\infty \leq {\mathcal D}_x + t{\mathcal D}_v,\qquad \sup_{i\in\{1,...,N\}}|v_i|_\infty \leq {\mathcal D}_v.
    \end{align*}
    
\smallskip
\noindent
The above solutions $(\vx^N,\vv^N)$ correspond through the formula \eqref{atomic} to unique measure-valued solutions $f^N$ of \eqref{mezo1} in the sense of Definition \ref{weaksol}. Finally, stability estimate in Proposition \ref{stab} and the fact that $f_0^N\to f_0$ in $({\mathcal P}_2(\R^{2d}),W_2)$ ensure that $f^N$ converge to a unique limit, denoted by $f$, in $C(0,T; ({\mathcal P}_2(\R^{2d}),W_2))$.

\medskip
In light of the above, and since uniqueness follows directly from the stability estimate, it remains to prove that $f$ is compactly supported satisfying bounds in item (ii) of Definition \ref{weaksol} and that equation \eqref{weaksol-1} is satisfied. We begin with the proof of item (ii) noting that for all $t\in[0,T]$,
\begin{align*}
    {\rm supp}f^N_t\subset B_{\frac{{\mathcal D}_x+t{\mathcal D}_v}{2}}\times B_{\frac{{\mathcal D}_v}{2}},
\end{align*}
where $B_R$ denotes the closed ball centered at $0$ with radius $R$.
The above inclusion is equivalent to saying that
\begin{align*}
    f^N_t(B_{\frac{{\mathcal D}_x+t{\mathcal D}_v}{2}}\times B_{\frac{{\mathcal D}_v}{2}}) = 1
\end{align*}
and, since for all $t\in[0,T]$ the measure $f^N_t\to f_t$ narrowly, all the involved balls are closed, Portmanteau theorem ensures that
\begin{align*}
    f_t(B_{\frac{{\mathcal D}_x+t{\mathcal D}_v}{2}}\times B_{\frac{{\mathcal D}_v}{2}}) \geq \limsup_{N\to\infty} f^N_t( B_{\frac{{\mathcal D}_x+t{\mathcal D}_v}{2}}\times B_{\frac{{\mathcal D}_v}{2}}) = 1
\end{align*}
for all $t\in[0,T]$. Therefore for all $t\in[0,T]$, we have
\begin{align*}
    {\rm supp}f_t\subset B_{\frac{{\mathcal D}_x+t{\mathcal D}_v}{2}}\times B_{\frac{{\mathcal D}_v}{2}}
\end{align*}
and the bound on the support of $f$ required in item (ii) of Definition \ref{weaksol} is proven.

\medskip
Finally, we aim to show that $f$ satisfies the weak formulation \eqref{weaksol-1}. Taking an arbitrary test function $\psi\in C^1_b([0,T]\times \R^{2d})$, compactly supported in $[0,T)$ we know that for each $N\in{\mathbb N}$ the measure-valued function $f^N$ satisfies \eqref{weaksol-1} with initial datum $f_0^N$. Thus, it suffices to prove the following convergences:

\begin{align*}
    \int_0^T\intr \partial_t\psi + v\cdot \nabla_x\psi\, \dd f^N_t\,\dd t &\to \int_0^T\intr \partial_t\psi + v\cdot \nabla_x\psi\, \dd f_t\,\dd t,\\
    \int_0^T\intr \nabla_v\psi\cdot F[f^N_t] \, \dd f^N_t\,\dd t &\to \int_0^T\intr \nabla_v\psi\cdot F[f_t] \, \dd f_t\,\dd t,\\
    \intr \psi(0,x,v) \,\dd f_0^N &\to \intr \psi(0,x,v) \,\dd f_0.
\end{align*}

\smallskip
\noindent
The first and the third are trivial, since the integrands are bounded and continuous and $f^N_t\to f_t$ in $W_2$ for all $t\in[0,T]$. As for the middle term involving the force $F[f_t]$, we use a classical weak-strong convergence argument. First we write
\begin{align*}
     \left|\int_0^T\intr \nabla_v\psi\cdot F[f^N_t] \, \dd f^N_t\,\dd t - \int_0^T\intr \nabla_v\psi\cdot F[f_t] \, \dd f_t\,\dd t\right|\leq\\
     \left|\int_0^T\intr \nabla_v\psi\cdot (F[f^N_t] -F[f_t])\, \dd f^N_t\,\dd t\right| + \left|\int_0^T\intr \nabla_v\psi\cdot F[f_t] \, (\dd f^N_t - \dd f_t)\,\dd t\right|.
\end{align*}

\smallskip
\noindent
By Lemma \ref{lip} we have
\begin{align*}
    \left|\int_0^T\intr \nabla_v\psi\cdot (F[f^N_t] -F[f_t])\, \dd f^N_t\,\dd t\right|\leq T |\nabla_v\psi|_\infty |F[f^N_t]-F[f_t]|_\infty \\
    \leq c_1 T |\nabla_v\psi|_\infty W_2(f^N_t,f_t) \longrightarrow 0.
\end{align*}

\smallskip
\noindent
Finally suppose that $F[f_t]$ is continuous and bounded as a function of $(t,x,v)$. Then by the characterisation of $W_2$ convergence that includes the narrow convergence tested with bounded continuous functions, we have
\begin{align*}
    \left|\int_0^T\intr \nabla_v\psi\cdot F[f_t] \, (\dd f^N_t - \dd f_t)\,\dd t\right| \to 0.
\end{align*}

\smallskip
\noindent
Altogether $f$ satisfies the weak formulation \eqref{weaksol-1} and the proof is finished.

\medskip
Thus it remains to ensure the continuity of $(t,x,v)\mapsto F[f_t](x,v)$. To this end let $(t_n,x_n,v_n)\to(t,x,v)$ and observe that
\begin{align*}
    |F[f_{t_n}](x_n,v_n)-F[f_{t}](x,v)| &\leq |F[f_{t_n}](x_n,v_n)-F[f_{t}](x_n,v_n)| \\
    &\qquad + |F[f_{t}](x_n,v_n)-F[f_{t}](x,v)|\\
    &\leq c_1 W_2(f_{t_n},f_t) + c_0(|x_n-x| + |v_n-v|)\to 0,
\end{align*}
where we used Lemma \ref{lip} and the fact that $f_{t_n}\to f_t$ in $W_2$.

\end{proof}

\section{Concluding remarks}\label{sec:conc}

To summarize the idea of the paper we need to return to system \eqref{qclosed}. 
Looking formally at the velocity equation, we see a system, which is piece-wise constant in time. The communication is not symmetric, and in fact it is often expected to occur only in one direction (particle being influenced by another, which it does not influence back). Our modification of the system makes it well posed from the mathematical viewpoint; the change of the communications becomes smooth. 

A very natural question concerns the asymptotics. It is well known that with $q>(N-1)/2$ there is a permanent communication between all particles, which results in global flocking -- all elements tend to the same speed at $t=+\infty$. However for smaller $q$
there is no such phenomenon. In fact a more natural question is if a consistent-in-time segmentation of the ensemble occurs or not. We shall call such phenomenon of large-time emmergence of isolated clusters {\it clustering}.

Note that for the fuzzy system, the question is virtually the same. For the large-time behavior, parameter $\sigma$
does not play any role. Let us formulate the question as the following conjecture.

\smallskip

{\sc Conjecture:}
{\it The set of neighbors ${\mathcal N}_i(t)$ to system (\ref{qclosed}) stabilizes in time into separate constant subsets, in particular we observe the velocity stabilisation, i.e.
\begin{equation}\label{conj}
    \mbox{ for each } i\in \{1,...,N\} \qquad
    \lim_{t\to \infty} \dot v_i(t)=0.
\end{equation}}

\smallskip 

In general we expect a negative answer to the above conjecture. On the torus one can find an example (essentially Example \ref{ex1} on the 2D torus). On the other hand in the case of the whole Euclidean space, our simulations indicate that the conjecture holds.

Finally we take another look at (\ref{qclosed}) with non-constant communication weight $\psi$. If the communication weight $\psi$ vanishes at infinity, we may lose the interesting combinatorical character of the problem. Consider
\begin{equation}\label{phi-sys}
\dot x_i = v_i, \qquad \dot v_i=-\frac{1}{q} \sum_{j \in \mathcal{N}_i} \psi(|x_i-x_j|)(v_i-v_j),
\end{equation}
where $\psi(\cdot)$ is a sufficiently smooth scalar function such that
$\psi(s) \to 0$ as $s\to \infty$. Note that such system is still ill-posed. The analysis of system (\ref{phi-sys}) seems to be the first step in proving our conjecture. However to find the suitable approach there is a need of a new viewpoint for such unsymmetrical models.

\appendix

\section{Atomic solutions}

The appendix consists of the following proposition, which states that the solutions of the fuzzy particle system \eqref{fuzzy}, treated as the empirical measures \eqref{atomic} constitute solutions to the kinetic equation \eqref{mezo1}.

\begin{prop}\label{chain}
Let $(\vx^N, \vv^N)$ be a solution to \eqref{fuzzy}. Then $f^N$ given in \eqref{atomic} satisfies \eqref{mezo1} in the sense of Definition \ref{weaksol}.
\end{prop}

\medskip

\begin{proof}
Given $\psi\in C_b^1([0,T)\times \R^{2d})$, by the chain rule, we have
\begin{equation*}
\begin{split}
    \frac{\dd}{\dd t}\intr \psi\, \dd f^N_t &= \frac{1}{N}\sum_{i=1}^N\frac{\dd}{\dd t}\psi(t,x_i^N(t),v_i^N(t))\\
    &= \frac{1}{N}\sum_{i=1}^N (\partial_t\psi)(t,x_i^N(t),v_i^N(t)) + \frac{1}{N}\sum_{i=1}^N v_i^N(t)\cdot\nabla \psi(t,x_i^N(t),v_i^N(t))\\
    &\qquad + \frac{1}{N}\sum_{i=1}^N\nabla_v\psi(t,x_i^N(t),v_i^N(t)) \frac{1}{\eta N}\sum_{j=1}^N\Bigg(v_i^N-\int_{B(x_i,R(x_i(t))}\phi(z-x_j)\, \dd z\ v_j^N\Bigg).
\end{split}
\end{equation*}

\smallskip
\noindent
To finish the proof we integrate the above equation over $[0,T)$ and observe that
\begin{align*}
        \intr \partial_t\psi\, \dd f^N_t &= \frac{1}{N}\sum_{i=1}^N (\partial_t\psi)(t,x_i^N(t),v_i^N(t)),\\
    \intr v\cdot\nabla \psi\, \dd f^N_t &= \frac{1}{N}\sum_{i=1}^N v_i^N(t)\cdot\nabla \psi(t,x_i^N(t),v_i^N(t)),\\
    \intr \nabla_v\psi F(f_N)\, \dd f^N_t &= \frac{1}{N^2}\frac{1}{\eta}\sum_{i,j=1}^N\nabla_v\psi(t,x^N_i(t),v_i^N(t))\Bigg(v_i^N - \int_{B(x_i^N,R(t,x_i^N))}\phi(z-x_j^N)\, \dd z\ v_j^N\Bigg).
\end{align*}
    
\end{proof}

\bibliographystyle{amsplain} 
\bibliography{main}

\end{document}